\documentclass[letterpaper, 10 pt, conference]{ieeeconf}
\IEEEoverridecommandlockouts

\usepackage{amssymb,amsmath}

\usepackage{amsthm}
\usepackage{mathtools}
\usepackage{multirow}
\usepackage{graphicx}
\usepackage{comment}
\usepackage[sort,compress,noadjust]{cite}
\usepackage[font=footnotesize]{subcaption}
\usepackage[font=footnotesize]{caption}

\usepackage{amsfonts}
 \usepackage{tabularx}
\usepackage{graphicx}
\usepackage{array}
\usepackage{float}
\usepackage{hhline}
\usepackage{algorithmic}
\usepackage[linesnumbered, ruled, vlined]{algorithm2e}
\usepackage{setspace}
\usepackage{colortbl}
\usepackage{arydshln}
\SetCommentSty{mycommfont}

\DeclareMathOperator*{\argmin}{arg\,min}

\usepackage{accents}

\usepackage{placeins}

\theoremstyle{plain}
\newtheorem{theorem}{Theorem}
\theoremstyle{plain}
\newtheorem{proposition}{Proposition}
\theoremstyle{plain}

\theoremstyle{plain}

\theoremstyle{definition}

\newtheorem{remark}{Remark}

\newtheoremstyle{definition}{}{}{}{}{\bfseries}{.}{.5em}{\thmname{#1}\thmnumber{ #2}\thmnote{ (#3)}}
\theoremstyle{definition}
\newtheorem{definition}{Definition}

\newcommand{\naturals}{\mathbb{N}}
\newcommand{\real}{\mathbb{R}}

\newcommand{\until}[1]{[#1]}
\newcommand{\map}[3]{#1:#2 \rightarrow #3}

\newcommand{\setdefb}[2]{\big\{#1 \; | \; #2\big\}}

\renewcommand{\bf}{\mathbf{f}} 
\newcommand{\bg}{\mathbf{g}}

\newcommand{\bk}{\mathbf{k}}

\newcommand{\bu}{\mathbf{u}}
\newcommand{\bv}{\mathbf{v}}

\newcommand{\bx}{\mathbf{x}}
\newcommand{\by}{\mathbf{y}}

\newcommand{\bA}{\mathbf{A}}

\newcommand{\bH}{\mathbf{H}}
\newcommand{\bI}{\mathbf{I}}

\newcommand{\bL}{\mathbf{L}}

\newcommand{\bP}{\mathbf{P}}

\newcommand{\bV}{\mathbf{V}}

\newcommand{\bZ}{\mathbf{Z}}

\newcommand{\balpha}{\boldsymbol{\alpha}}

\newcommand{\bzero}{\mathbf{0}}
\newcommand{\bPsi}{\boldsymbol{\Psi}}
\newcommand{\bxi}{\boldsymbol{\xi}}

\newcommand{\Cc}{\mathcal{C}}

\newcommand{\Ec}{\mathcal{E}}

\newcommand{\Kc}{\mathcal{K}}

\newcommand{\Sc}{\mathcal{S}}
\newcommand{\Vc}{\mathcal{V}}

\newcommand{\Wc}{\mathcal{W}}

\usepackage[colorlinks,urlcolor=blue]{hyperref}
\usepackage[capitalize]{cleveref}
\crefformat{equation}{(#2#1#3)}
\Crefformat{equation}{Equation~(#2#1#3)}
\Crefname{equation}{Equation}{Eqs.}

\title{\LARGE \textbf{High-Order Matrix Control Barrier Functions:\\ Well-Posedness and Feasibility via Matrix Relative Degree}}

\author{Samuel G. Gessow$^1$, Pio Ong$^2$, Aaron D. Ames$^2$, and Brett T. Lopez$^1$
\thanks{$^1$Authors are with VECTR Laboratory, University of California, Los Angeles, Los Angeles, CA, USA, {\tt\footnotesize \{sgessow, btlopez\}@ucla.edu}}
\thanks{$^2$Authors are with the Department of Mechanical and Civil Engineering, California Institute of Technology, Pasadena, CA 91125, USA, {\tt\footnotesize \{pioong, ames\}@caltech.edu}}
}

\begin{document}
\maketitle
\thispagestyle{empty}
\begin{abstract}
    Control barrier functions (CBFs) provide an effective framework for enforcing safety in dynamical systems with scalar constraints. However, many safety constraints are more naturally expressed as matrix-valued conditions, such as positive definiteness or eigenvalue bounds --- scalar formulations introduce potential nonsmoothness that complicates analysis. \emph{Matrix control barrier functions (MCBFs)} address this limitation by directly enforcing matrix-valued safety constraints. Yet for constraints where the control input does not appear in the first derivative, high-order formulations are required. While such extensions are well understood in the scalar case, they remain largely unexplored in the matrix case. This paper develops \emph{high-order matrix control barrier functions (HOMCBFs)} and establishes conditions ensuring well-posedness and feasibility of the associated constraints, enabling enforcement of matrix-valued safety constraints for systems with high-order dynamics. We further show that, using an optimal-decay HOMCBF formulation, forward invariance can be ensured while requiring control only over the minimum eigenspace. The framework is demonstrated on a double-integrator localization problem by enforcing positive definiteness of the information matrix.
\end{abstract}
\section{Introduction}
Control barrier functions (CBFs) provide an effective framework for enforcing safety constraints on dynamic systems \cite{ames2017control,ames2019control}, with applications that include autonomous driving, mobile robotics, aerial systems, among others \cite{hsu2024safety}.
In many cases, safety specifications can naturally be expressed as the zero-super level set of a scalar function, for which necessary and sufficient conditions for forward invariance are well understood.
However, there is a growing class of applications in which safety is more naturally described by a matrix-valued function \cite{PO-BC-LS-JC:23-auto, gessow2024information,ong2026matrix}. 
A common approach in such settings is to reduce matrix-valued conditions to scalar constraints on their eigenvalues~\cite{PO-BC-LS-JC:23-auto, gessow2024information}.
While this enables the use of standard scalar CBF tools, it introduces significant technical challenges, as eigenvalues are inherently nonsmooth. Within this scalarization framework, there are two main approaches: one directly handles the nonsmoothness using nonsmooth barrier functions~\cite{glotfelter2017nonsmooth} as in~\cite{PO-BC-LS-JC:23-auto} while the other avoids nonsmooth points by imposing additional conditions, such as eigenvalue separation~\cite{gessow2024information}.  
Given these complications, a more natural framework that avoids scalarization is of interest.

Recently, matrix control barrier functions (MCBFs) have been proposed to handle matrix-valued safety constraints~\cite{ong2026matrix}, such as positive definiteness and eigenvalue bounds.
Critically, MCBFs avoid differentiating the eigenvalues of matrix-valued functions, thereby circumventing difficulties from nonsmoothness and improving computation times.
By avoiding the additional conditions required by nonsmoothness, MCBFs also reduce the conservatism of scalar eigenvalue-based formulations, which can shrink the admissible safe set.

Despite these advantages, the theoretical foundations of MCBFs remain incomplete. One central issue is \emph{well-posedness}, namely whether a given matrix-valued constraint can be enforced through Lie-derivative-based inequalities.
This property is fundamentally tied to the geometry of the boundary of the safe set. In the scalar setting, well-posedness admits a simple and complete characterization: the constraint function must have zero as a regular value~\cite{ames2019control} (also related to practical set in~\cite{blanchini2008set}). This condition is a key assumption for many results in the CBF literature.
Complementary to this is the issue of \emph{feasibility}. Even when a set is well-posed, enforcing the associated barrier condition requires verifying a linear matrix inequality (LMI), which is difficult to check.

Beyond well-posedness and feasibility, a key missing component in the theory of MCBFs is the treatment of high relative degree constraints. In many systems, safety constraints are not directly influenced by the control input, but only through high-order system dynamics.
In the scalar setting, this challenge is well understood, with techniques such as high-order CBF (HOCBF)~\cite{xiao2019control} followed by optimal decay CBFs~\cite{gurriet2020scalable,zeng2021safety,ong2025properties} which reduce well-posedness to only a condition on the boundary of the safe set.
In contrast, no analogous framework exists for matrix-valued constraints. 

In this work, we develop a framework for matrix-valued constraints in high-order dynamics settings. 
A key challenge in doing so is ensuring that the safe set from a matrix safety function is enforceable.
To address this, we develop a matrix analogue of relative degree, along with an optimal-decay term, to enable high-order MCBFs. 
This matrix analogue is derived from an analysis of well-posedness of sets defined by matrix-valued constraints, ensuring that the resulting conditions are meaningful.
In particular, the proposed conditions are based on the ability to influence the minimum eigenvalue of the matrix-valued function, making them closely aligned with the conditions from the scalar case.
Forward invariance can be ensured without requiring explicit control over non-minimal eigenvalues by utilizing optimal-decay.
We further develop a tractable sufficient condition for matrix relative degree, facilitating its verification. 
We demonstrate HOMCBFs for localization safety on a double-integrator system.

\label{sec:introduction}
\section{Preliminaries}
\subsection{Notation}
The set of real numbers is denoted $\real$ with $\real^n$ denoting the $n$-dimensional Euclidean space and $\real_{\geq0}$ the set of non-negative scalars.
The set of real symmetric $p\times p$ matrices is denoted $\mathbb{S}^p$, and the cone of positive definite symmetric matrices is denoted $\mathbb{S}^p_+$.
We say a continuous function $\alpha:\real\to\real$ is an extended class-$\mathcal{K}$ function $(\alpha \in \mathcal{K}^e)$ if it is strictly monotonically increasing and satisfies $\alpha(0)=0$.
The eigenvalues of matrix $\bA$ are $\lambda(\bA)$ and the smallest eigenvalue is $\lambda_{\mathrm{min}} (\bA) = \min \, \lambda(\bA)$. 
The Lie derivative of a scalar function $\bx \mapsto h(\bx)$ along a vector field $\bx \mapsto \bf(\bx)$ is denoted $L_\bf h(\bx)= \nabla h( \bx)^\top \bf(\bx)$. Similarly, the Lie derivative of a matrix-valued function $\bx \mapsto \bH(\bx)\in\mathbb{S}^p$ along a vector field $\bx \mapsto \bf(\bx)$ is denoted $\bL_\bf \bH(\bx)$, and it is equivalent to a matrix of element-wise Lie derivative, i.e., $[\bL_\bf \bH(x)]_{ij}=\bL_f(H_{ij})(x)$, $i,j\in\{1,\dots,p\}$.

\subsection{Control Barrier Functions}
Control barrier functions (CBFs) provide a framework for enforcing state constraints in nonlinear systems through feedback control. Consider the control-affine system:
\begin{equation}\label{sys:ctrl_affine}
    \dot \bx = \bf(\bx) + \bg(\bx) \bu
\end{equation}
where $\bx \in \real^n$ is the state and $\bu\in\real^m$ is the control input. We assume in this paper that the drift vector field $\map{\bf}{\real^n}{\real^n}$ and the control matrix $\map{\bg}{\real^n}{\real^{n\times m}}$ are smooth.  

A safety specification is encoded through a set $\Sc \subseteq \real^n$. 
The control objective is to ensure that system trajectories $t\mapsto\bx(t)$ remain in $\Sc$ for all time $t\geq 0$. 
This requirement is equivalent to rendering $\Sc$ \textit{forward invariant} under the closed-loop dynamics. 
Achieving this requires the set $\Sc$ to be \textit{control invariant}, i.e., for every initial condition $\bx_0\in\Sc$, there exists a control input $t\mapsto \bu(t)$  such that the ensuing trajectory is contained within $\Sc$. 
Since safety constraints are typically prescribed by the task and may not naturally be control invariant, it is common to instead construct a subset $\Cc\subseteq \Sc$ that is control invariant. 
To this end, CBFs provide a tool for certifying control invariance of such sets.

\begin{definition}
Let $\Sc \subseteq \real^n$ be a safety constraint set. A set $\Cc$ is said to be \emph{safe} for \cref{sys:ctrl_affine} if $\Cc\subseteq \Sc$ and $\Cc$ is control invariant.
\end{definition}

In the CBF framework, the set $\Cc$ is typically described as the $0$-superlevel set of a scalar-valued function 
$\map{h}{\real^n}{\real}$:
\begin{equation}
    \Cc = \setdefb{\bx\in\real^n}{h(\bx)\geq 0}.
\end{equation}
To ensure forward invariance of $\Cc$, the function $h$ must remain nonnegative at all time $t\geq 0$. This leads to the notion of a control barrier function.

\begin{definition}[cf.~\cite{ames2017control}]\label{def:CBF}
Let $\Cc \subset \mathbb{R}^n$ be the $0$-superlevel set of a continuously differentiable function 
$h : \mathbb{R}^n \rightarrow \mathbb{R}$. 
The function $h$ is a {\emph{control barrier function}} for \cref{sys:ctrl_affine} if there exists an extended class-$\mathcal{K}^e$ function $\alpha$ such that for each $\bx \in \Cc$ there exists a $\bu \in \real^m$ such that
\begin{equation}
\label{eq:CBF_cond}
     \dot h(\bx,\bu) \triangleq L_\bf h(\bx) +  L_{\bg} h(\bx) \bu > - \alpha(h(\bx)).
\end{equation}
\end{definition}

Although the CBF condition in \cref{def:CBF} only directly guarantees the pointwise existence of a control input, the condition actually implies the existence of a continuous feedback controller $\map{\bk}{\real^n}{\real^m}$ that renders the differential inequality $\dot h(\bx,\bk(\bx)) \geq -\alpha(h(\bx))$. In particular, a widely used controller is the CBF-based quadratic program (CBF-QP) which selects the control input closest to a given nominal controller $\bx\mapsto\bk_{\rm nom}(\bx)$ as
\begin{equation}\label{eq:cbf_qp}
\begin{aligned}
\bk(\bx) = \argmin_{\bu\in\real^m} \quad & \|\bu - \bk_{\rm nom}(\bx)\|^2 \\
\text{s.t.} \quad 
& L_\bf h(\bx) +  L_{\bg} h(\bx) \bu \ge -\alpha(h(\bx)).
\end{aligned}
\end{equation}
Under the CBF condition, the strict inequality in \cref{def:CBF} ensures that the CBF-QP is continuous. Hence, the feedback $\bu=\bk(\bx)$ renders $\dot h \geq -\alpha(h)$ along the closed-loop trajectories, ensuring via the comparison lemma that $h$ remains nonnegative so that $\Cc$ is forward invariant.

\subsection{High-Order CBF}
While the CBF framework above provides a convenient condition for enforcing safety constraints, it requires that the control input appear in the derivative of the barrier function~$h$. In particular, the CBF constraint depends on the term $L_\bg h(\bx)$, and for some systems, this term may vanish identically, i.e., $L_\bg h(\bx)\equiv \bzero$. Such situations arise, for example, in strict-feedback systems or when the constraint function has a \textit{relative degree} greater than one.

In this case, the CBF condition reduces to
\begin{equation}
    L_\bf h(\bx) + \alpha(h(\bx))\geq 0.
\end{equation}
Due to the absence of the control input, the satisfaction of the barrier condition above depends entirely on the natural system dynamics. Nevertheless, the independence on $\bu$ allows a new construction of a barrier function candidate:
$$
\psi_2(\bx) = L_\bf h(\bx) + \alpha(h(\bx)),
$$
which we aim to keep nonnegative. If the control input appears in the derivative of $\psi_2$, then the barrier condition can be enforced at this level.

More generally, define the sequence of functions recursively as
\begin{equation}\label{eq:HOCBF}
\begin{aligned}
    \psi_1(\bx) &= h(\bx),\\
    \psi_{i+1}(\bx) &= L_\bf \psi_{i}(\bx) +\alpha_{i}(\psi_{i}(\bx)),~ i\in\until{r-1},
\end{aligned}
\end{equation}
where $\{\alpha_i\}_{i=1}^{r-1}$ are smooth extended class-$\mathcal K^e$ functions. The integer $r\in\naturals$ denotes the first level where the derivative of $\psi_r$ depends on $\bu$, i.e., $L_\bg \psi_r (\bx) \not \equiv \bzero$.

\begin{definition}[cf.~\cite{xiao2022high}]
\label{def:HOCBF}
Given a smooth function $\map{h}{\real^n}{\real}$, let the functions $\{\psi_i\}_{i\in\until{r}}$ be defined recursively as in~\eqref{eq:HOCBF}. For each $i\in\until{r}$, define the set $\Cc_i$ corresponding to $\psi_i$ as
\[
\Cc_i = \{\bx \in \mathbb{R}^n : \psi_i(\bx) \ge 0\}.
\]
The function $h$ is a \textit{high-order control barrier function} (HOCBF) for~\eqref{sys:ctrl_affine} if there exists an extended class-$\Kc^e$ function $\alpha_r$ such that for each $\bx\in\cap_{i\in\until{r}}\Cc_i$ there exists a $\bu \in \real^m$ such that
\begin{equation}
\label{eq:HOCBF_cond}
     \dot \psi_r(\bx,\bu) \triangleq L_\bf \psi_r(\bx) +  L_{\bg} \psi_r(\bx) \bu > - \alpha_r(\psi_r(\bx)).
\end{equation}
\end{definition}
Similar to CBFs, HOCBFs can be used to enforce safety constraints by designing a continuous controller $\map{\bk}{\real^n}{\real^m}$ that satisfies $\dot \psi_r(\bx,\bk(\bx))\geq -\alpha_r(\psi_r(\bx))$. In practice, this can be implemented using a CBF-QP similar to~\eqref{eq:cbf_qp}, but with the HOCBF constraint replacing the CBF constraint. Note, however, due to the recursive construction of the functions $\{\psi_i\}$ required to ensure that the control input can influence the barrier condition, forward invariance is only ensured on the intersection $\cap_{i\in\until r}\Cc_i$. The HOCBF condition enforces positivity of the lowest-level function $\psi_r$, and the recursive relations in~\eqref{eq:HOCBF} only ensure that this positivity propagates upward to each $\psi_i$ (and ultimately $h$) if each $\psi_i$ is initially nonnegative. 

The HOCBF construction is closely related to the notion of relative degree from feedback linearization. 
\begin{definition}
A smooth function $h: \mathbb{R}^n \rightarrow \mathbb{R}$ is said to have relative degree $r\in \mathbb{N}$ with respect to the input $\bu$ for \cref{sys:ctrl_affine} at $\bx\in \mathbb{R}^n$ if
\begin{enumerate}
    \item $L_\bg L_\bf^k h(\bx) = 0, \quad \forall k \in \{0, \ldots, r-2\},$
    \item $L_\bg L_\bf^{r-1} h(\bx) \neq 0.$
\end{enumerate}
Moreover, $h$ is said to have relative degree $r$ on a set $\mathcal{E}\subseteq \mathbb{R}^n$ if it has relative degree $r$ with respect to $\bu$ for all $\bx\in \mathcal{E}$.
\label{def:rel_deg}
\end{definition}
From the definition, if the function $h$ has relative degree $r$ on the intersection $\cap_{i\in\until r} \Cc_i$, then $L_\bg \psi_r(\bx)= L_\bg L_\bf^{r-1} h(\bx) \neq 0$, ensuring that $h$ meets the conditions in~\cref{def:HOCBF} to be a HOCBF. This relative-degree assumption is the primary condition adopted in the original HOCBF formulation in~\cite{xiao2022high} as a way to verify condition~\cref{eq:HOCBF_cond}.

\subsection{Matrix Control Barrier Function}
Matrix control barrier functions (MCBFs) generalize the scalar CBF framework by replacing the scalar barrier function $\map{h}{\real^n}{\real}$ with a symmetric matrix-valued function $\map{\bH}{\real^n}{\mathbb{S}^p}$. In this setting, the safe set of interest is better represented with a positive semidefinite constraint:
\begin{equation}
\Cc = \setdefb{\bx\in\real^n}{\bH(\bx)\succeq \bzero}.
\label{eq:safe_set}
\end{equation}
Aligned with this safe set definition, MCBFs are developed using the ordering induced by the positive semidefinite cone.

\begin{definition}[cf.~\cite{ong2026matrix}]
Let $\Cc \subset \mathbb{R}^n$ be defined by a continuously differentiable function $\map{\bH}{\real^n}{\mathbb{S}^p}$ as in~\eqref{eq:safe_set}. The function $\bH$ is a \emph{matrix control barrier function} (MCBF) for \cref{sys:ctrl_affine} if there exists an extended class-$\mathcal{K}^e$ function $\alpha$ such that, for each $\bx\in\Cc$, there exists a $\bu\in \mathbb{R}^m$ satisfying
\begin{equation}
    \dot{\bH}(\bx,\bu) \triangleq\bL_\bf\bH(\bx)+\sum_{j=1}^m \bL_{\bg_j}\bH(\bx) u_j \succ -\balpha(\bH(\bx)),
\end{equation}
with $\balpha(\cdot)$ defined spectrally (see \cite[Eq. (23)]{ong2026matrix}), and $\bg_j$ is a column of $\bg$. 
\label{def:MCBF}
\end{definition}
Similar to CBFs, MCBFs ensure the possibility of designing a continuous controller $\map{\bk}{\real^n}{\real^m}$ that satisfies $\dot \bH(\bx,\bk(\bx))\succeq -\balpha(\bH(\bx))$, which guarantees $\bH$ remains positive semidefinite along system trajectories~\cite{ong2026matrix}. In practice, such controllers can be constructed using optimization-based safety filters analogous to the CBF-QP~\cref{eq:cbf_qp}. In particular, replacing the scalar barrier constraint with the matrix inequality yields a semidefinite program (CBF-SDP), which is continuous under the strict inequality of the barrier condition of~\cref{def:MCBF}.
A useful specific version of the MCBF is the exponential MCBF where $\alpha$ is chosen to be linear, i.e., $\alpha(r) = c_\alpha r$ with $c_\alpha>0$, in which case $\balpha(\bH) = c_\alpha \bH$ and no spectral decomposition of $\bH$ is required.

Despite the similarities with the scalar case, the existing MCBF framework does not address safety constraints whose matrix barrier function has \textit{relative degree} higher than one. The development of such high-order MCBFs (HOMCBFs), together with a notion of relative degree for matrix-valued functions, is of interest in this paper.

\section{High-Order Matrix CBFs}
While MCBFs provide a natural framework for enforcing matrix-valued safety constraints, their applicability is limited to cases where the control input appears in the derivative of the matrix function $\bH$. In particular, when $\bL_{\bg_i}\bH(\bx)\equiv \bzero$ for all $i\in\until{m}$, the control input does not directly influence the barrier condition, and high-order constructions are required. In this section, we therefore develop high-order matrix control barrier functions (HOMCBFs).

Analogous to the scalar case, we define the sequence of functions recursively as
\begin{align}\label{eq:HOMCBF}
    \bPsi_1(\bx)&= \bH(\bx),\\
    \bPsi_{i+1}(\bx) &= \bL_\bf\bPsi_i(\bx) + \balpha_i(\bPsi_i(\bx)),~i\in\until{r-1},
\end{align}
with extended class-$\mathcal{K}^e$ functions $\{\balpha_i\}_{i=1}^{r-1}$, and with $r\in\naturals$ being the first level where there exists $j\in\until m$ such that $\bL_{\bg_j}\bPsi_r\not \equiv \bzero$. We now formally define a HOMCBF.

\begin{definition}\label{def:HOMCBF}
    Given a smooth function $\map{\bH}{\real^n}{\mathbb S^p}$, let the functions $\{\bPsi_i\}_{i\in\until{r}}$ be defined recursively as in \cref{eq:HOMCBF}. For each $i\in\until r$, define the set $\Cc_i$ corresponding to $\bPsi_i$ as
    \begin{equation}
        \label{eq:HO_safe_set}
        \Cc_i = \setdefb{\bx\in\real^n}{\bPsi_i(\bx)\succeq 0}.
    \end{equation}
    The function $\bH$ is a \emph{high-order matrix control barrier function} for system~\eqref{sys:ctrl_affine} if there exists an extended class-$\Kc^e$ function $\balpha_r$ such that, for each $\bx\in\cap_{i\in\until r} \Cc_i$, there exists a $\bu\in\real^m$ satisfying
    \begin{equation}\label{eq:HOMCBF_cond_strict}
    \dot{\bPsi}_r(\bx,\bu) \triangleq \bL_\bf\bPsi_r(\bx)+\sum_{j=1}^m \bL_{\bg_j}\bPsi_r(\bx) u_j \succ -\balpha_r(\bPsi_r(\bx)).
    \end{equation}
\end{definition}

Similar to the scalar case, the HOMCBF condition facilitates the design of a feedback controller to render the set intersection~$\cap_{i\in\until r}\Cc_i$ forward invariant.
\begin{theorem}
\label{thm:HOMCBF}
    Consider the control-affine system~\eqref{sys:ctrl_affine} with sets $\{\Cc_i\}_{i\in\until r}$ defined as in~\eqref{eq:HO_safe_set}. If $\bH$ is a HOMCBF for~\eqref{sys:ctrl_affine}, then any continuous feedback controller $\map{\bk}{\real^n}{\real^m}$ satisfying 
    \begin{equation}\label{eq:HOMCBF_cond}
    \dot\bPsi_r(\bx,\bk(\bx))\succeq -\balpha_r(\bPsi_r(\bx)),
    \end{equation}
    for all $\bx$ in an open neighborhood of $\cap_{i\in\until r}\Cc_i$, renders the intersection $\cap_{i\in\until r} \Cc_i$ forward invariant for the closed-loop system under the state feedback $\bu=\bk(\bx)$.
    
    In particular, the CBF-SDP controller given by
    \begin{align}
        \label{eq:HOMCBF-SDP}
        \bk(\bx) = \argmin_{\bu\in\real^m} \quad &\|\bu-\bk_{\rm nom}(\bx)\|^2 \\
                  \textup{s.t.} \quad &\dot \bPsi_r(\bx,\bu) \succeq -\balpha_r(\bPsi_r(\bx))\nonumber
    \end{align}
    is continuous on an open neighborhood of $\cap_{i\in\until r}\Cc_i$, and therefore renders the intersection $\cap_{i\in\until r} \Cc_i$ forward invariant for the closed-loop system.
\end{theorem}
\begin{proof}
    Denote with $\Ec\supset \cap_{i\in\until r}\Cc_i$ the open neighborhood where~\eqref{eq:HOMCBF_cond} holds. Under the state feedback $\bu=\bk(\bx)$, suppose there exists a trajectory $t\mapsto \bx(t)$ with an initial condition $\bx_0\in\cap_{i\in\until r}\Cc_i$ such that the trajectory leaves the intersection $\cap_{i\in\until r}\Cc_i$. Then because $\Ec$ is an open neighborhood of the intersection $\cap_{i\in\until r} \Cc_i$, there exists time $t^*$ such that $\bx(t)\in\Ec$ for all time $t \in [0,t^*]$ and $\bx(t^*)\in \Ec \setminus \cap_{i\in\until r}\Cc_i$.
    
    The HOMCBF condition ensures that
    $$
    \frac{d}{dt}\bPsi_r(\bx(t)) = \dot \bPsi_r(\bx(t),\bk(\bx(t)))\succeq -\balpha_r(\bPsi_r(\bx(t))) 
    $$
    for all time $t$ such that  $\bx(t) \in \Ec$. This implies $\bPsi_r(\bx(t)) \succeq \bzero$ for time $[0,t^*]$ as both $\bk$ and $\bPsi_r(\bx(t))$ are continuous, see also~\cite[Prop. 1]{ong2026matrix} for the matrix comparison argument. By definition~\eqref{eq:HOMCBF} of $\bPsi_r$, we deduce $\dot \bPsi_{r-1}(\bx(t)) \succeq -\balpha_{r-1}(\bPsi_{r-1}(\bx(t)))$. Again, by comparison arguments, we have $\bPsi_{r-1}(\bx(t)) \succeq \bzero$ for time $[0,t^*]$. Repeating this process recursively shows that $\bPsi_i(\bx(t))\succeq \bzero$ for all $i\in\until r$ and time $t\in[0, t^*]$. This however implies $\bx(t) \in \cap_{i\in\until{r}}\Cc_i$ for all time $t \in [0,t^*]$, which is a contradiction with $\bx(t^*)\in \Ec \setminus \cap_{i\in\until r}\Cc_i$. Therefore, the intersection $\cap_{i\in\until{r}}\Cc_i$ is forward invariant.

    It remains to prove the statement regarding the HOMCBF-SDP controller. By Definition~\ref{def:HOMCBF}, the HOMCBF-SDP is strictly feasible on $\cap_{i\in\until r}\Cc_i$. From continuity of the functions involved in the constraint, strict feasibility also holds on some neighborhood of $\cap_{i\in\until r}\Cc_i$. Since the CBF-SDP satisfies Slater's condition, it is continuous there; see \cite[Thm. 1]{ong2026matrix} for the complete proof of this statement.
\end{proof}

The above result directly parallels the scalar HOCBF framework, with the online SDP enforcing only the terminal condition~\eqref{eq:HOMCBF_cond} and the recursive construction of the functions $\{\bPsi_i\}_{i\in\until r}$ ensuring that the positivity of the highest-order function propagates to the original constraint.
In particular, the safe set verified is the intersection $\cap_{i\in\until r}\Cc_i$, which is a subset of the set $\Cc_1$ associated with $\bH$. 
If $\Cc_1\subseteq\Sc$, then forward invariance of $\cap_{i\in\until r}\Cc_i$ guarantees safety.

\section{Matrix Relative Degree Notions}

In contrast to the scalar case, the HOMCBF condition involves a linear matrix inequality, making it significantly more challenging to verify in practice. In particular, characterizing a notion analogous to relative degree for the matrix setting that guarantees the barrier condition~\eqref{eq:HOMCBF_cond_strict} is nontrivial. To address this challenge, we develop our intuition on how control inputs influence matrix-valued CBF through studying the simplest single-integrator dynamics, in order to develop a notion of relative degree for matrix-valued constraints.

\subsection{Scalar Case Revisited}
We briefly recall the scalar case to highlight the key property underlying relative degree that enables verification of CBFs. For a scalar constraint $\map{h}{\real^n}{\real}$ with relative degree one on $\Cc$, the existence of a control direction $\bg_i$ such that $L_{\bg_i}h(\bx)\neq 0$ is sufficient to influence the evolution of $h$ and enforce the barrier condition. 
In particular, $u_i$ can be selected sufficiently large, with an appropriate sign, to satisfy~\eqref{eq:CBF_cond}. 
This is precisely what relative degree one ensures.
A similar conclusion holds for HOCBFs with high relative degree on $\cap_{i\in\until r} \Cc_i$.

An analogous condition for matrix-valued constraints, following the same logic, would require that $\bL_{\bg_i}\bH(\bx)$ be definite for some $i\in\until m$, so that the corresponding input direction can be scaled to enforce the matrix inequality~\eqref{eq:HOMCBF_cond_strict}. However, this requirement is overly conservative. Even relaxing this condition to only require that the set $\{\bL_{\bg_i}\bH(\bx)\}_{i\in\until m}$ spans a definite matrix remains restrictive, as such conditions are rarely satisfied in practice.

\subsection{Matrix Constraint Well-Posedness}
\label{sec:single_int}
We seek a more fundamental characterization of when a matrix-valued constraint can be rendered feasible. As a starting point, we consider the simplest setting of single-integrator dynamics and formulate a notion of well-posedness for sets defined by matrix constraints in this case. This notion will serve as the basis for later developing notions of matrix relative degrees.

Under single-integrator dynamics,~$\dot\bx = \bu$, the control input directly influences the state evolution, providing maximal flexibility in shaping trajectories. As such, a natural notion of \textit{well-posedness} is that the set defined by the matrix-valued constraint $\bPsi(\bx)\succeq\bzero$ can be enforced through an appropriate choice of control input in this setting.
To determine whether the matrix $\bPsi$ can be maintained positive semidefinite, it is not necessary to control the behavior of $\bPsi$ in all directions. Rather, only the directions associated with the smallest eigenvalue are critical, as these are the directions along which positive semidefiniteness is first lost. To this end, we propose the following mathematical definition for well-posedness.

\begin{definition}
\label{def:well-posed}
    Given a continuously differentiable matrix-valued function $\map{\bPsi}{\real^n}{\mathbb S^p}$, let $\lambda_{\min}(\bPsi(\bx))$ denote the smallest eigenvalue of $\bPsi(\bx)$, and let $\bV(\bx)$ be a matrix whose columns form an orthonormal basis for the eigenspace associated with $\lambda_{\min}(\bPsi(\bx))$. The set defined by $\bPsi(\bx)\succeq \bzero$ is \emph{well-posed} if, for each $\bx$ such that $\lambda_{\min}(\bPsi(\bx))=0$, there exists $\bu\in\real^n$ such that
    \begin{equation}
        \sum_{j=1}^n 
        \bV(\bx)^\top 
        \frac{\partial \bPsi}{\partial x_j}(\bx)\,
        \bV(\bx)\, u_j 
        \succ \bzero.
        \label{eq:well-posed}
    \end{equation}
\end{definition}
There is a clear interpretation of the above condition when the smallest eigenvalue~$\lambda_{\min}(\bPsi(\bx))$ is simple. 
In this case, the mapping $\lambda_{\min} \circ \bPsi$ is differentiable, and condition~\eqref{eq:well-posed} reduces to requiring that its gradient $\frac{\partial (\lambda_{\min}\circ\bPsi)}{\partial \bx}\neq 0$ whenever $\lambda_{\min}(\bPsi(\bx))=0$. 
This is precisely the regular value condition, which ensures that the boundary of the set $\{\bx \mid \bPsi(\bx)\succeq 0\}$ is well-defined\footnote{Also related to the notion of practical set proposed in \cite[Ch. 4.2.1]{blanchini2008set} in connection to Nagumo's theorem}.

When $\lambda_{\min}$ is not simple (i.e., repeated), the condition is not identical to the regular-value condition. 
Nevertheless, \cref{eq:well-posed} remains sufficient because $\bV$ spans the active
minimum eigenspace and the projected inequality ensures that every zero-eigenvalue direction can be controlled toward the interior of the positive semidefinite cone. 
It can therefore be viewed as a nonsmooth regularity condition.

\begin{theorem}\label{thm:well-posed}
    Consider the set $\Cc =\{\bx \mid \bPsi(\bx)\succeq \bzero\}$ constructed from a continuously differentiable function $\map{\bPsi}{\real^n}{\mathbb S^p}$. If the set defined by the constraint $\bPsi(\bx)\succeq \bzero$ is well-posed in the sense of \cref{def:well-posed}, then under the single-integrator dynamics $\dot\bx = \bu$, there exists a smooth feedback controller $\map{\bk}{\real^n}{\real^n}$ such that $\bu=\bk(\bx)$ renders the set $\Cc$ forward invariant under the closed-loop dynamics.
\end{theorem}
\begin{proof}
     Let $\bx\mapsto \bu^*(\bx)$ denote the control given from~\eqref{eq:well-posed} for all $\bx\in\partial\Cc$, i.e., where $\lambda_{\min}(\bPsi(\bx))=0$.  For $\bx\in{\rm int}(\Cc)$, let $\bu^*(\bx)$ be defined arbitrarily, e.g., $\bu^*(\bx)=\bzero$. The control choice $\bu^*$ may not be a continuous function, and this proof aims to construct $\bk$ from it.

     The proof follows the arguments underlying Artstein's smooth controller construction~\cite[Thm. 4.1]{ZA:83}. 
     We first show that if there exists $\bu^*(\bx)$ satisfying~\eqref{eq:well-posed} at $\bx$, the same control~$\bu^*(\bx)$, evaluated at $\bx$, also satisfies the condition on some neighborhood $\Wc(\bx)$ of $\bx$. 
     
     We note that the condition~\eqref{eq:well-posed} is equivalent to 
     \begin{equation}
     \label{eq:well-posed_eigen}
     \bv^\top\sum_{j=1}^n \frac{\partial \bPsi}{\partial x_j}(\bx)u_j^*(\bx)\bv > 0
     \end{equation}
     for all $\bv$ in the image of $\bV(\bx)$, i.e., in the eigenspace associated with the eigenvalue $\lambda_{\min}(\bPsi(\bx)) = 0$. Since eigenvalues are Lipschitz continuous, there exists a neighborhood $\Wc_0(\bx)$ of $\bx$ such that all other nonzero eigenvalues remain separated from the eigenvalues clustered around zero.  In this neighborhood, the multiplicity of $\lambda_{\min}(\bx)=0$ may decrease, splitting into multiple eigenvalues. Let
     $\bx\mapsto\Vc(\bx)$ denote the direct sum of the eigenspaces associated with this cluster of eigenvalues, then the total spectral projector $\bx'\mapsto\bP(\bx')$ onto $\Vc(\bx')$ is continuous (adopting smoothness from $\bPsi$) at each $\bx'\in \Wc_0(\bx)$~\cite[Ch. 2 Sec. 1]{TK:66}. Here, we note the projector satisfies $\bP(\bx')\bv = \bv$ for any $\bv\in\Vc(\bx')$. Using this property, the well-posedness condition~\eqref{eq:well-posed}, at $\bx$ in particular, can be equivalently rewritten as
     $$
     \hat\bv^\top \bP(\bx)^\top\sum_{j=1}^n\frac{\partial \bPsi}{\partial x_j}(\bx)u_j^*(\bx)\bP(\bx)\hat\bv \geq \eta > 0
     $$
     for all unit vectors $\hat\bv\in\Vc(\bx)$ such that $\|\hat\bv\|=1$. The constant $\eta>0$ is the minimum bound achieved from the compactness of the unit sphere $\|\hat\bv\|=1$. Consequently, from continuity of $\bP$ and $\frac{\partial\bPsi}{\partial \bx}$, there exists a smaller neighborhood $\Wc(\bx)\subset \Wc_0(\bx)$ such that we can hold $\bu^*(\bx)$ constant and obtain
     $$
     \bv^\top \bP(\bx')^\top\sum_{j=1}^n\frac{\partial \bPsi}{\partial x_j}(\bx')u_j^*(\bx)\bP(\bx')\bv > 0
     $$
     for all $\bx'\in\Wc(\bx)$ and all $\bv\in\Vc(\bx')\setminus\{\bzero\}$. 
     In other words, $\bu=\bu^*(\bx)$ satisfies the condition on $\Wc(\bx)$. 
     We then construct $\bk$ via a smooth partition of unity subordinate to a countable subcover of $\{\Wc(\bx)\}_{\bx\in\Cc}$; see~\cite[Thm. 4.1]{ZA:83}. 
     Since~\eqref{eq:well-posed} is convex in $\bu$, the resulting $\bk$ also meets the condition and hence renders $\Cc$ forward invariant (see sufficiency in~\cite[Ch. 4.2.1]{blanchini2008set}).
\end{proof}

Theorem~\ref{thm:well-posed} establishes that the well-posedness condition guarantees the existence of a control law that renders the set $\Cc = \{\bx \mid \bPsi(\bx)\succeq \bzero\}$ forward invariant under single-integrator dynamics. This result formalizes a minimal requirement under which matrix-valued constraints can be enforced through control inputs, and serves as a baseline notion of feasibility. Building on this insight, we next develop a notion of relative degrees for matrix-valued functions.

\begin{remark}
In light of Theorem~\ref{thm:well-posed}, we can make similar arguments as in the scalar case in \cite[Ch. 4.2.1]{blanchini2008set} to establish that forward invariance can be deduced from Lie derivative inequalities. In contrast, without well-posedness, barrier conditions alone are generally insufficient to guarantee forward invariance. We omit a full  technical treatment of this fact due to space limitations. Note also that our Definition~\ref{def:HOMCBF} proposes~\eqref{eq:HOMCBF_cond_strict}  with a strict inequality, which implicitly ensures well-posedness.
\end{remark}

\subsection{Matrix Relative Degrees}
The intuition gained from \cref{sec:single_int}, on necessitating control over the eigenspace associated with the minimum eigenvalue, can be extended to general control affine systems. Specifically, we generalize \cref{def:well-posed} to define a notion of matrix relative degree. 

\begin{definition}
\label{def:matrix_relative_degree}
A smooth function $\bH:\mathbb{R}^n\to\mathbb{S}^p$ is said
to have \emph{matrix relative degree} $r\in\mathbb{N}$ with respect to
the input $\bu$ for \cref{sys:ctrl_affine} at
$\bx\in\mathbb{R}^n$ if the following hold. Let $\bV_r(\bx)$ be a
matrix whose columns form an orthonormal basis for the eigenspace of
$\bPsi_r(\bx)$ associated with
$\lambda_{\min}(\bPsi_r(\bx))$.
\begin{enumerate}
    \item For each $j\in\until m$ and each
    $k\in\{0,\dots,r-2\}$,
    \begin{equation}\label{eq:matrix_reldeg_lower}
        \bL_{\bg_j}\bL_\bf^k\bH(\bx)=\bzero.
    \end{equation}

    \item There exists $\bu\in\mathbb{R}^m$ such that
    \begin{equation}\label{eq:matrix_reldeg_r}
        \bV_r(\bx)^\top\!\left(
        \sum_{j=1}^m
        \bL_{\bg_j}\bL_\bf^{r-1}\bH(\bx)\,u_j
        \right)\bV_r(\bx)\succ\bzero.
    \end{equation}
\end{enumerate}
Moreover, $\bH$ is said to have matrix relative degree $r$ on
$\mathcal{E}\subseteq\mathbb{R}^n$ if it has matrix relative degree
$r$ for all $\bx\in\mathcal{E}$.
\end{definition}

While Definition~\ref{def:matrix_relative_degree} provides a precise characterization, it may be difficult to verify directly as it involves solving an LMI. The following result provides a simple sufficient condition that can be checked pointwise.
\begin{proposition}
\label{prop:matrix_relative_degree_sufficient}
Given a smooth matrix-valued function $\map{\bH}{\real^n}{\mathbb S^p}$, let $\{\bPsi_i\}_{i\in\until{r}}$ be defined recursively as in \cref{eq:HOMCBF}. Let $\lambda_{\min}(\bPsi_r(\bx))$ denote the smallest eigenvalue of $\bPsi_r(\bx)$, and let $\bV_r(\bx)$ be a matrix whose columns form an orthonormal basis for the eigenspace associated with $\lambda_{\min}(\bPsi_r(\bx))$.
If there exists an index
    $j^*\in\until m$ such that the matrix
\begin{equation}\label{eq:sufficient_matrix_reldeg}
    \bV_r(\bx)^\top \bL_{\bg_{j^*}}\bPsi_r(\bx)\,\bV_r(\bx)
\end{equation}
is either positive definite or negative definite, then $\bH$ has matrix relative degree $r\in\naturals$ at $\bx$.
\end{proposition}
\begin{proof}
The recursive construction of $\{\bPsi_i\}_{i\in\until{r}}$ ensures that $\bL_{\bg_j}\bPsi_i =\bL_{\bg_j}\bL_\bf^{i-1}\bH(\bx) =\bzero$ for all $j\in\until m$ and all $i\in\until{r-1}$, which is precisely~\eqref{eq:matrix_reldeg_lower}. In addition, by selecting $u_j=0$ for all $j\neq j^*$ and $u_{j^*}$ large enough with an appropriate sign, we can ensure $\bV_r(\bx)^\top\!\left(\sum_{j=1}^m \bL_{\bg_j}\bPsi_r(\bx)\,u_j\right)\bV_r(\bx)\succ \bzero$, which is equivalent to~\eqref{eq:matrix_reldeg_r}. 
\end{proof}
\cref{prop:matrix_relative_degree_sufficient} provides a condition that does not require solving for a control input $\bu$. Instead, the condition is a simple algebraic expression that is more convenient to verify in practice. In addition, it is less conservative than requiring $\bL_{\bg_j}\bPsi_r(\bx)$ itself to be positive definite or negative definite, since definiteness is only required after projection onto the eigenspace associated with the minimum eigenvalue of $\bPsi_r(\bx)$.  On the other hand, the condition does not guarantee direct control over the remaining eigendirections. However, as we show next, these directions can be accommodated through the optimal-decay CBF (OD-CBF) framework.

\section{Optimal Decay HOMCBF}
\label{sec:optimal_decay}
While matrix relative degree provides a useful sufficient characterization for enforcing matrix-valued safety constraints, requiring it to hold uniformly over the entire safe set can be restrictive in practice, as is also observed in the scalar case~\cite[Sec. 4B]{ong2025properties}. Nevertheless, by leveraging ideas from the OD-CBF framework in~\cite{ong2025properties}, we can restrict this requirement to only the boundary of the safe set. Moreover, this method has the additional benefit of accommodating non-minimal eigendirections.
We introduce an auxiliary variable in the following definition to formalize this idea.
\begin{definition}\label{def:OD-HOMCBF}
    Given a smooth function $\map{\bH}{\real^n}{\mathbb S^p}$, let the functions $\{\bPsi_i\}_{i\in\until{r}}$ be defined recursively as in \cref{eq:HOMCBF}. For each $i\in\until r$, define the set $\Cc_i$ corresponding to $\bPsi_i$ as
    \begin{equation}
        \Cc_i = \setdefb{\bx\in\real^n}{\bPsi_i(\bx)\succeq 0}.
    \end{equation}
    The function $\bH$ is an \emph{optimal decay high-order matrix control barrier function} (OD-HOMCBF) for system~\eqref{sys:ctrl_affine} if there exists an extended class-$\Kc^e$ function $\balpha_r$ such that, for each $\bx\in\cap_{i\in\until r} \Cc_i$, there exists a $\bu\in\real^m$ and $\omega \in \real_{\ge 0}$ satisfying
    \begin{equation}\label{eq:OD-HOMCBF_cond_strict}
    \dot{\bPsi}_r(\bx,\bu) \succ -\omega \balpha_r(\bPsi_r(\bx)).
\end{equation}
\end{definition}
The auxiliary variable $\omega$ allows the decay rate of the non-critical eigenvalues of $\bPsi_r$ to be relaxed, making it feasible even when control authority is only required to influence the minimum eigenvalues at points where it is zero.
This additional variable can be easily incorporated into the SDP safety filter by weighting $\omega$ into the preexisting cost function, as formalized in the following theorem.
\begin{theorem}
\label{thm:OD-HOMCBF}
    Consider the control-affine system~\eqref{sys:ctrl_affine} with sets $\{\Cc_i\}_{i\in\until r}$ defined as in~\eqref{eq:HO_safe_set}. If $\bH$ is an OD-HOMCBF for~\eqref{sys:ctrl_affine}, then any continuous feedback controller $\map{\bk}{\real^n}{\real^m}$ and $\map{\theta}{\real^n}{\real_{\ge 0}}$ satisfying
    \begin{equation}
    \dot\bPsi_r(\bx,\bk(\bx))\succeq -\theta(\bx)\balpha_r(\bPsi_r(\bx)),
    \end{equation}
    for all $\bx$ in an open neighborhood of $\cap_{i\in\until r}\Cc_i$, renders the intersection $\cap_{i\in\until r} \Cc_i$ forward invariant for the closed-loop system under the state feedback $\bu=\bk(\bx)$.
    
    In particular, the CBF-SDP controller given by
    \begin{align}
        \label{eq:OD-HOMCBF-SDP}
        \bk(\bx) = \argmin_{\substack{\bu \in \mathbb{R}^m \\ \omega \in \mathbb{R}_{\ge 0}}} & \quad \|\bu-\bk_{\rm nom}(\bx)\|^2  + p_\omega (\omega - \theta_d)^2\\
                  \textup{s.t.} \quad &\dot \bPsi_r(\bx,\bu) \succeq - \omega\balpha_r(\bPsi_r(\bx))\nonumber \\
                  &\omega\ge \theta_d \nonumber
    \end{align}
    where $p_\omega>0$ is the penalty weight and $\theta_d\geq0$ is the desired decay scaling, is continuous on an open neighborhood of $\cap_{i\in\until r}\Cc_i$, and therefore renders the intersection $\cap_{i\in\until r} \Cc_i$ forward invariant for the closed-loop system.
\end{theorem}
\begin{proof}
    The proof is identical to that of \cref{thm:HOMCBF}. The only modification is
that the terminal inequality is now
\[
\dot{\bPsi}_r(\bx,\bk(\bx)) \succeq -\theta(\bx)\balpha_r(\bPsi_r(\bx)),
\]
with $\theta(\bx)\ge \theta_d$. Since this is the same matrix comparison inequality as
in \cref{thm:HOMCBF}, up to a nonnegative state-dependent scaling of the
class-$\Kc$ term, the same recursive argument implies successively that
$\bPsi_r,\bPsi_{r-1},\dots,\bPsi_1$ remain positive semidefinite along the
closed-loop trajectories. Hence $\bigcap_{i\in[r]} \Cc_i$ is forward invariant.
\end{proof}

We can now prove that with an OD-HOMCBF and uniform matrix relative degree on the boundary safe sets where $\bPsi_r \not\succ \bzero$ we can get forward invariance of the safe set under smooth feedback control.
\begin{theorem}
Consider the control-affine system~\eqref{sys:ctrl_affine} with sets $\{\Cc_i\}_{i\in[r]}$ defined as in \eqref{eq:HO_safe_set}. 
If $\bH$ has matrix relative degree $r$ on an open neighborhood of the points $\bx\in\bigcap_{i\in[r]}\Cc_i$ satisfying $\lambda_{\min}(\bPsi_r(\bx))=0$, then $\bH$ is an OD-HOMCBF on $\bigcap_{i\in[r]}\Cc_i$. 
Consequently, there exists a continuous feedback controller $\bk:\mathbb{R}^n\to\mathbb{R}^m$ such that $\bigcap_{i\in[r]}\Cc_i$ is forward invariant for the closed-loop system under $\bu=\bk(\bx)$.
\end{theorem}
\begin{proof}
Let $\bx \in \bigcap_{i\in[r]} \Cc_i$ be such that $\bPsi_r(\bx)\succ \bzero$, then for any $\bv\in\real^p$, $\bv^\top\bPsi_r(\bx)\bv >0$. Therefore, for an arbitrary $\bu\in\real^m$, 
$$
\bv^\top\dot\bPsi_r(\bx,\bu)\bv > -\omega \bv^\top\balpha_r(\bPsi_r(\bx))\bv,
$$
can be satisfied for any $\bv\in\real^p$ with an $\omega > \small\frac{-\lambda_{\min}(\dot \bPsi_r(\bx))}{\alpha_r(\lambda_{\min}(\bPsi_r(\bx)))}$.
Therefore, there exists $\bu\in\real^m$ and $\omega\geq 0$ satisfying \eqref{eq:OD-HOMCBF_cond_strict}.

Let $\bx \in \bigcap_{i\in[r]} \Cc_i$ be such that $\bPsi_r \not\succ \bzero$ but $\bPsi_r \succeq \bzero$. Here, if all eigenvalues of $\bPsi_r(\bx)$ are zero, then condition~\eqref{eq:matrix_reldeg_r} from matrix relative degree imposes the existence of $\bu^\ast$ such that $\bv^\top\dot\bPsi_r(\bx,\bu^\ast)\bv>0$ for any $\bv$, and \eqref{eq:OD-HOMCBF_cond_strict} holds with $\bu^*$ and any $\omega\geq 0$.  If there exist nonzero eigenvalues, let $\lambda'(\bPsi_r(\bx))$ denote the smallest nonzero eigenvalue and $\bV=[\bV_0\;\bV_1]$ be the kernel and positive-eigenspace basis of $\bPsi_r(\bx)$ where $\bV^\top\dot\bPsi_r(\bx,\bu^\ast)\bV=[\bZ_{ij}]_{i,j\in\{0,1\}}$ where $\bu^*$ ensures $\bZ_{00}\succ \bzero$. Then, the control $\bu^*$ with $\omega >  \frac{\lambda_{\max}\left(\bZ_{10}\bZ_{00}^{-1}\bZ_{01}-\bZ_{11}\right)}{\alpha_r\left(\lambda'(\bPsi_r(\bx))\right)}$ satisfies~\eqref{eq:OD-HOMCBF_cond_strict}.

Since $\bH$ is an OD-HOMCBF, \cref{thm:OD-HOMCBF} applies and implies that the
closed-loop system under $\bu = \bk(\bx)$ where $\bk(\bx)$ can come from \cref{eq:OD-HOMCBF-SDP} renders
$\bigcap_{i\in[r]} \Cc_i$ forward invariant.
\end{proof}
Unlike in the scalar case, matrix relative degree alone is not able to certify $\bH$ is a HOMCBF, as a HOMCBF requires control over all the eigendirections, not just the minimum one.
This showcases the contribution of the auxiliary variable $\omega$, and the OD-CBF framework to the matrix case as it not only fixes the issues that necessitated checking relative degree over the entire set, but also establishes forward invariance with only direct control over the minimum eigenvalue. 
Under the OD-HOMCBF framework, we are able to ensure forward invariance of the safe set, by first validating the relative degree over the necessary part of the boundary, which can be done point-wise with {\cref{prop:matrix_relative_degree_sufficient}}, and then the SDP controller from $\cref{eq:OD-HOMCBF-SDP}$ can be used to render the set forward invariant and safe.
This contribution is demonstrated on the problem of ensuring localization safety.

\section{Application in Safe Localization}
\subsection{Problem}
Modern localization algorithms estimate the systems state over sensor measurements in real time.
There are several methods for estimating the state using the measurement, e.g., Kalman filtering and state observers.
However, recent works in the robotics literature have shown the superiority of optimization-based approaches that involve solving a nonlinear least squares (NLS) problem in an online fashion \cite{leutenegger2015keyframe,barfoot2024state,chen2023direct}. Regardless of the method, autonomy, planning, and control, all rely on the accuracy of state estimates, and therefore, localization failures are detrimental to safe operation.
Nevertheless, localization failures still occur in practice despite significant advances in sensing and computation.
These failures often occur rapidly with little warning, making recovery difficult. Motivated by this, methods such as \cite{costante2016perception,falanga2018pampc,gessow2024information} are developed to prevent localization failures.
The method in \cite{gessow2024information} directly addresses the issue by employing a CBF to ensure the optimization has a unique minimum. The safety condition considered is fundamentally a matrix condition, but was previously addressed via scalar methods, adding unnecessary complication.

This problem can be formally expressed by considering the control affine system~\eqref{sys:ctrl_affine} with a nonlinear measurement model $\mathfrak{m} : \mathbb{R}^n \rightarrow \mathbb{R}^p$ as
\begin{equation}
    \label{eq:meas_model}
    \by = \mathfrak{m}(\bx).
\end{equation}
We assume that the measurement map $\mathfrak{m}$ is injective, so that the
state $\bx$ can be uniquely reconstructed from the (noise free) measurement $\by$.
In particular, given measurements $\by \in \mathbb{R}^p$, the state $\bx$ can be obtained via solving the NLS problem:
\begin{equation}
\label{eq:nl}
        \bx =\argmin_{\bxi \, \in \, \mathbb{R}^n} \,  \|\by - \mathfrak{m}(\bxi)\|_{\Sigma(\bxi)}^2 \triangleq \argmin_{\bxi \, \in \, \mathbb{R}^n} J(\bxi),
\end{equation}
where $\Sigma : \mathbb{R}^n \rightarrow \mathbb{S}^{p}$ is a state-dependent positive definite matrix that represents the level of confidence in different components of $\by$.
We allow $\Sigma$ to be state-dependent as a way to capture measurement degradation/dropout; this will be discussed more in \cref{sec:results}.

At any given state $\bx$, the optimization is locally convergent around the true state $\bx$ if the cost function $J$ has a positive definite Hessian $\nabla^2_{\bxi} J(\bxi)$ when evaluated at $\bxi=\bx$. Thus, the collection of states where the NLS problem is well-behaved can be described as
\begin{equation}
    \Sc = \setdefb{\bx\in\real^n}{\bH(\bx)\triangleq \nabla^2_{\bxi} J(\bx)-\lambda_s\bI \succeq \bzero},
    \label{eq:ICBF_safe_set}
\end{equation}
where $\lambda_s > 0$ specifies the desired level of strong convexity. The set $\Sc$ represents the safety constraint associated with the state estimation task.
Since the safe set is described by a matrix-valued function, it naturally fits into the MCBF framework.

\begin{proposition}
\label{prop:icbf_forward_invariance}

Let $\nabla^2_{\bxi}J(\bx)\in\mathbb{S}^n$ denote the Hessian of the
nonlinear least squares cost function~\cref{eq:nl}, evaluated at
$\bxi=\bx$, and let $\Sc$ be defined as in
\cref{eq:ICBF_safe_set}. If
$\bH=\nabla^2_{\bxi}J(\bx)-\lambda_s\bI$ is an OD-HOMCBF as in
\cref{def:OD-HOMCBF}, with sets
$\{\Cc_i\}_{i\in[r]}$, then we call $\bH$ a \textit{matrix information
control barrier function} (MI-CBF), and the following hold:
\begin{enumerate}
    \item The set $\bigcap_{i\in[r]}\Cc_i\subseteq\Sc$ is control
    invariant and therefore safe.
    \item The CBF-SDP safety filter from
    \cref{eq:OD-HOMCBF-SDP} associated with $\bH$ is continuous on a
    neighborhood of $\bigcap_{i\in[r]}\Cc_i$ and renders this set
    forward invariant.
\end{enumerate}
\label{prop:MICBF}
\end{proposition}

This MI-CBF formulation has some key advantages over the scalar version proposed in \cite{gessow2024information}.
In the scalar method, there is no guarantee that the minimum eigenvalue is smooth, so additional assumptions and machinery are required to ensure that the safe set could be rendered forward invariant with a CBF.
Specifically it required smoothing to ensure that the Hessian was an analytic function of one variable.
This forced the safety filter to have a closed form solution to make sure the control was sufficiently smooth.
Alternatively, one could impose additional CBFs to ensure that the eigenvalues all remained unique, however, this approach added more parameters to tune as well as affected the resulting trajectory and increased computation.
In contrast, with a matrix formulation, a simple SDP-based safety filter can be used.
\subsection{Results}
\label{sec:results}
We demonstrate the method on the two examples of \cite{gessow2024information}, both of which admit no smooth naive scalar CBF.
On these examples we show that the MI-CBF ensures safety while being easier to implement than the approach taken in \cite{gessow2024information}, as it no longer requires explicit expressions for the eigenvalue derivatives. 
The double integrator has matrix relative degree two, since the Hessian depends on position but not velocity, making its first matrix Lie derivatives identically zero. 
We enforce the safe set with an OD-HOMCBF, at each step the state is computed by solving the NLS problem via gradient descent, with either range-only or bearing-only measurements. The estimated state and Hessian are then used by the safety filter to adjust an LQR controller driving the system to a desired location. 
The filter is an SDP requiring only matrix Lie derivatives of the Hessian, which also reduces the number of parameters to tune.
The measurement model $\mathfrak{m}(x)$ predicts the rate of change of the measurements. The system position is $(p_x,\,p_y)$ with beacons at $({b^k_{x}},\,b^k_{y})$ for $k=1,2,3$.

\textit{Range-Only Measurements.}
The measurement model {for range-only beacons} is $\mathfrak{m}_k(p_x ,p_y )=\sqrt{(p_x-b^k_{x})^2+(p_y-b^k_{y})^2}$ for $k=1,2,3$.
The covariance matrix $\Sigma(p_x,p_y)$ for each measurement was chosen to be $\Sigma(p_x,p_y)_{k,k}={1+\exp({m_k(p_x,p_y)-10})}$ to capture measurement dropout due to distance from the $k^{\text{th}}$ beacon. 
The off-diagonal entries were set to zero.
In \cref{fig:distance_trajectories}, gray scale denotes the value of $\lambda_{\text{min}}$ with the red line at $\lambda_s=5$ to visualize the safe region.
\begin{figure}[t!]
\centering
  \subfloat[Trajectory,]{\includegraphics[trim=20 0 10 0, clip, width=.44\columnwidth]{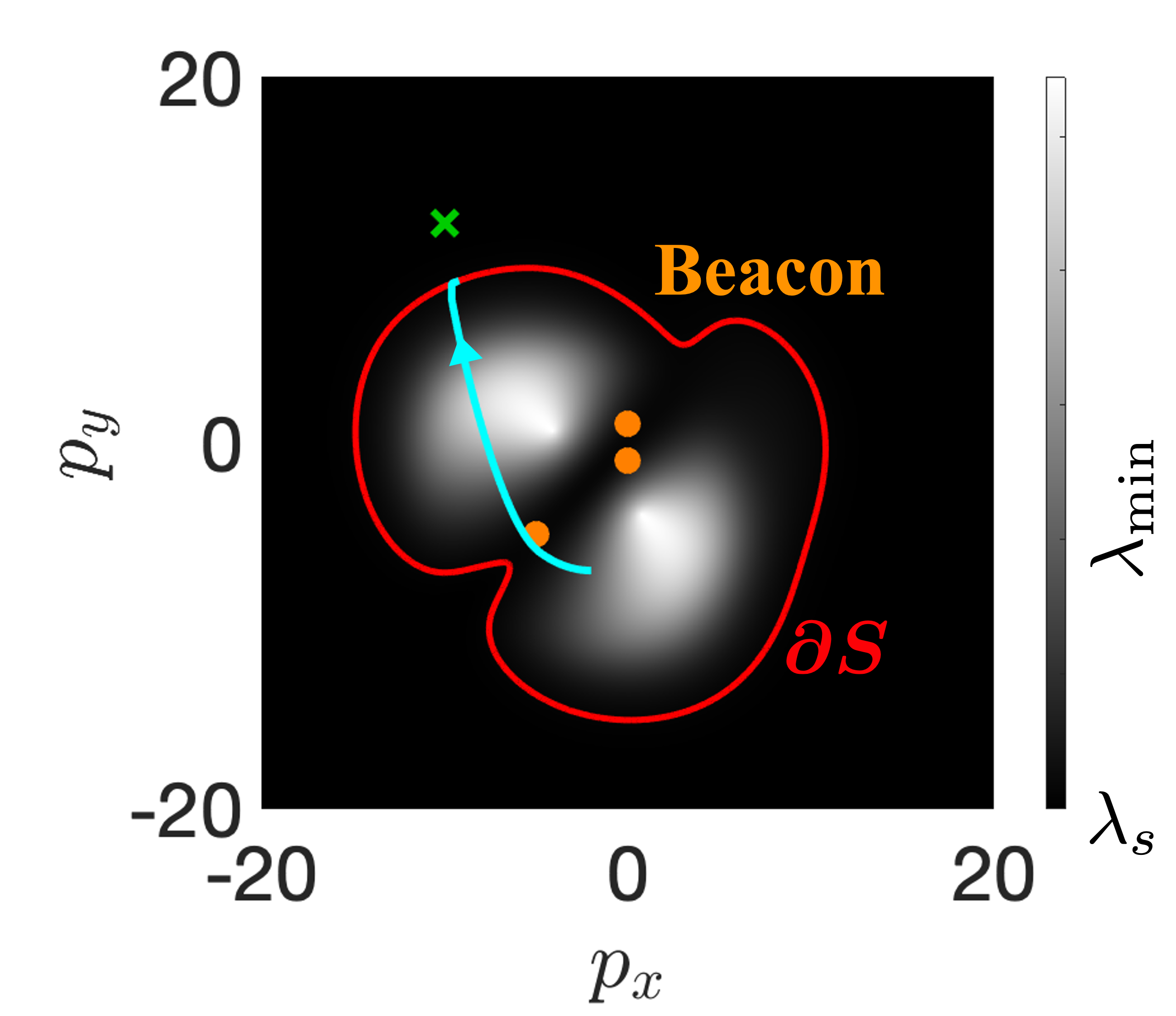}\label{fig:distance_trajectories}}
  \hspace{0.2em}
  \subfloat[Minimum eigenvalue of barrier.]{\includegraphics[trim=5 0 40 0, clip, width=.49\columnwidth]{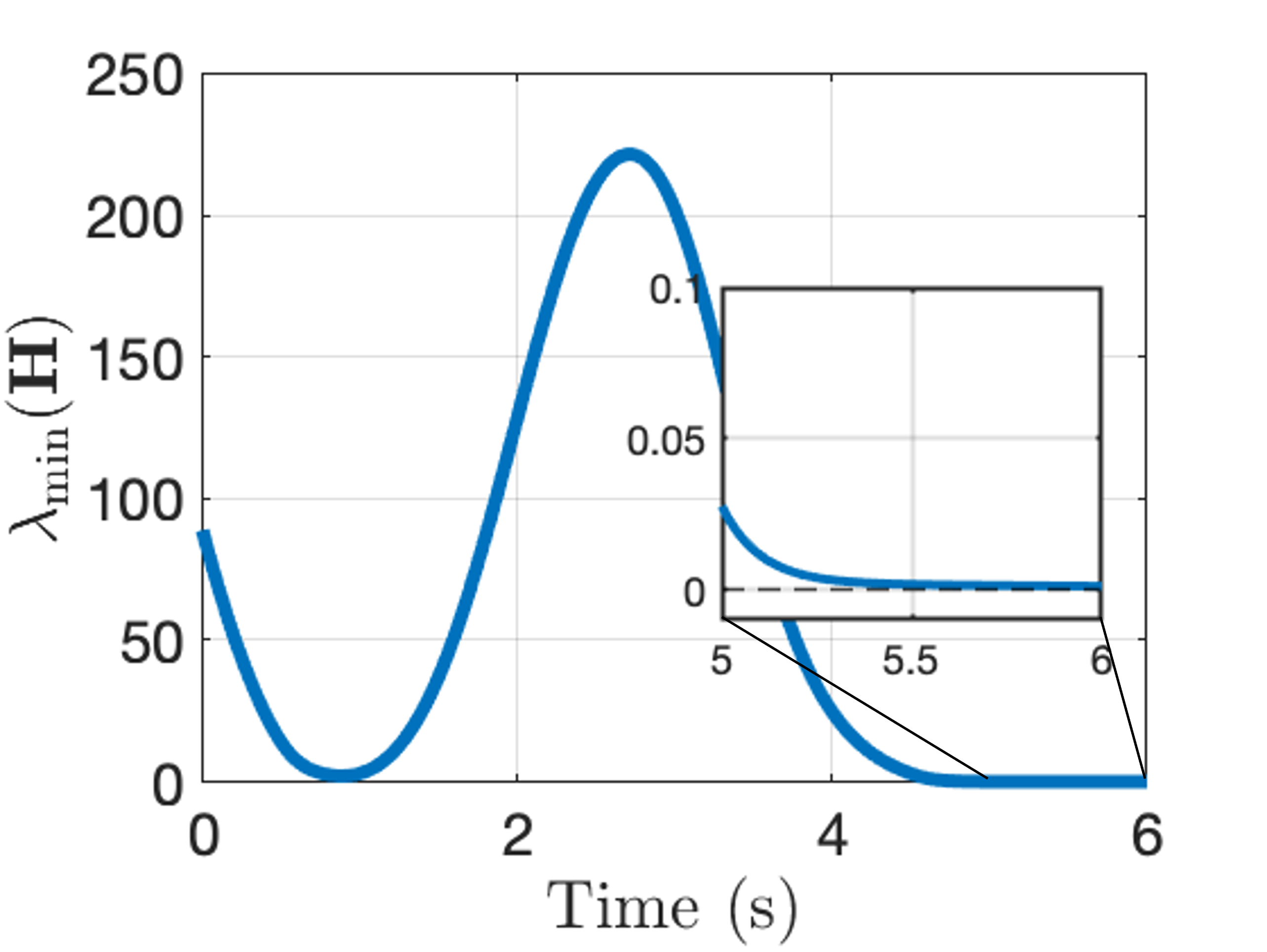}\label{fig:distance_barrier}}
  \caption{(a): Trajectory using range-only measurements; the safe region is indicated as inside the red line and the desired state with the green $\times$.  \\
  (b): The barrier value of $\lambda_{\min} (\bH)$}
    \label{fig:distance_results}
    \vskip -0.1in
\end{figure}

\textit{Bearing-Only Measurements.}
Next, we demonstrate the method using heading-only measurements. 
The measurement model is
$\mathfrak{m}_k(p_x ,p_y )=\tan^{-1}\left((p_y-b^k_y)/(p_x-b^k_x)\right)$ for $k=1,2,3$ and the weighting matrix is $\Sigma(p_x,p_y)=I$ since the eigenvalues naturally decayed as the distance from the beacons increased, thereby eliminating the need to model dropout separately.
The safe region, where $\lambda_s\geq0.015$, is shown in \cref{fig:angle_trajectory} by the red line.
\begin{figure}[t!]
\centering
  \subfloat[Trajectory,]{\includegraphics[trim=20 0 10 0, clip, width=.44\columnwidth]{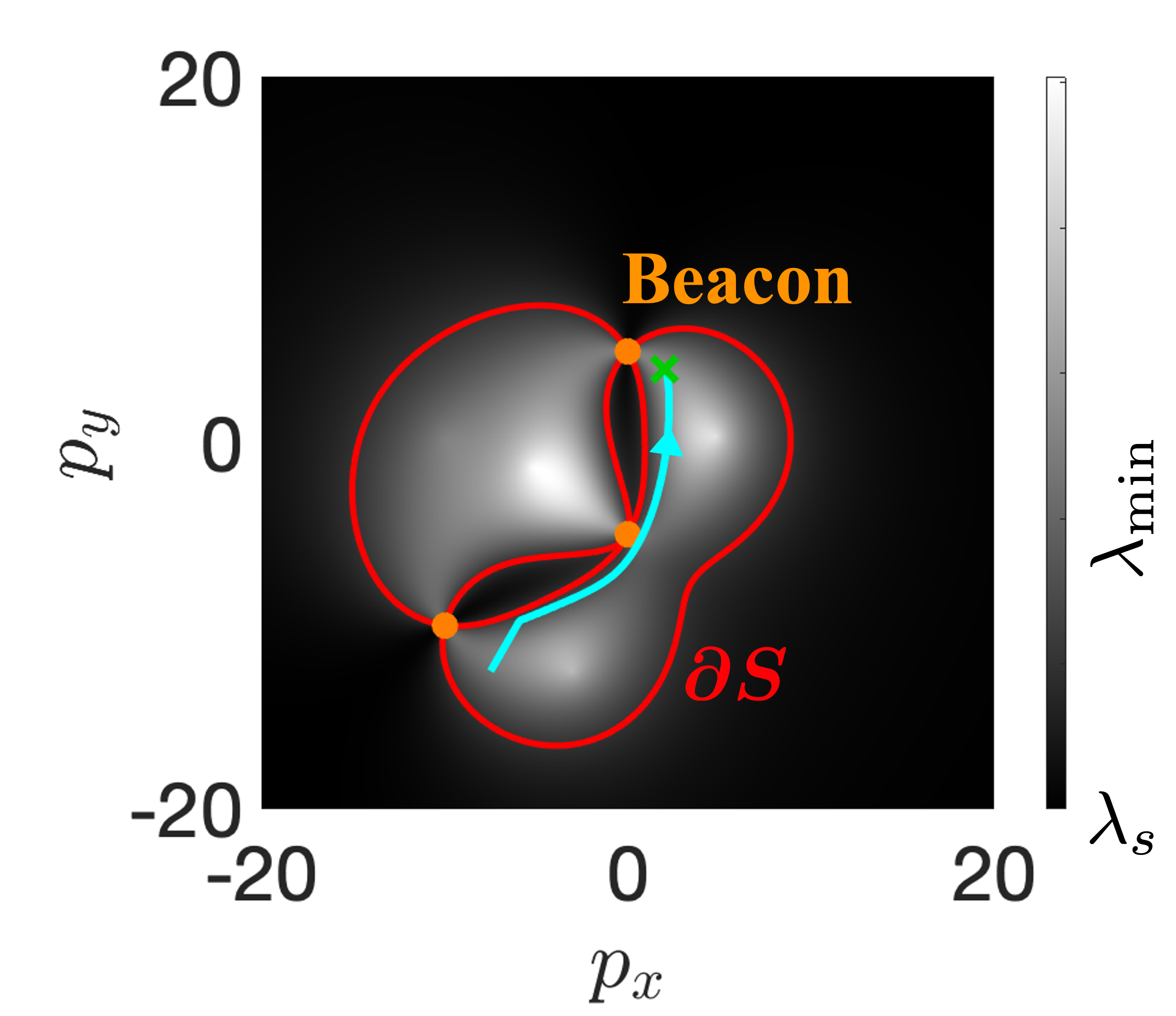}\label{fig:angle_trajectory}}
  \hspace{0.2em}
  \subfloat[Minimum eigenvalue of barrier.]{\includegraphics[trim=5 0 40 0, clip, width=.49\columnwidth]{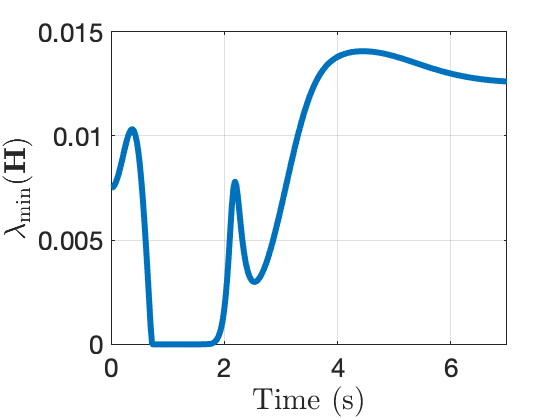}\label{fig:angle_barrier}}
  \caption{(a): Trajectory using heading-only measurements; the safe region is indicated as inside the red line and the desired state with the green $\times$.  \\
  (b): The barrier value of $\lambda_{\min} (\bH)$.}
    \label{fig:angle_results}
    \vskip -0.25in
\end{figure}
\section{Conclusion}
In this work we extended MCBFs to high-order systems by developing HOMCBFs and a corresponding notion of matrix relative degree.
A key challenge is that well-posedness and feasibility of matrix-valued safe sets require more care than in the scalar case.
We showed that these properties can be characterized through the ability to influence the minimum eigenspace, developed a tractable condition for matrix relative degree, and established that the OD-HOMCBF formulation guarantees forward invariance while only requiring fixed matrix relative degree on the boundary.
We demonstrated the framework on the problem of ensuring localization safety.
Future work includes investigating mixed relative degree MCBFs, as well as additional applications for HOMCBFs.
\bibliographystyle{IEEEtran}

\bibliography{references}
\end{document}